\documentclass[12pt,twoside]{article}
\usepackage{amssymb,amsmath,amsthm,latexsym}
\usepackage{amsfonts}
\usepackage{enumerate}
\usepackage{graphicx}

\newtheorem{thm}{Theorem}[section]

\newtheorem{cor}[thm] {Corollary}
\newtheorem{lem} [thm]{Lemma}

\theoremstyle{definition} 

\newtheorem{ex}[thm]{Example}

\newtheorem{defn}[thm]{Definition}
\numberwithin{equation}{section}

\newcommand\bfB{\mathbf{B}}
\newcommand\diag{\operatorname{diag}}

\begin{document}
	\label{'ubf'}
	\setcounter{page}{1} 
	
	\markboth {\hspace*{-9mm} \centerline{\footnotesize \sc
			Signed Distance in Signed Graphs}
	}
	{ \centerline {\footnotesize \sc
			Hameed, Shijin, Soorya, Germina, Zaslavsky}
	}
	\begin{center}
		{
			{\huge \textbf{Signed Distance in Signed Graphs
				}
			}
			
			\bigskip
			
			Shahul Hameed K \footnote{\small  Department of
				Mathematics, K M M Government\ Women's\ College, Kannur - 670004,\ Kerala,  \ India.  E-mail: shabrennen@gmail.com}
			Shijin T V \footnote{\small Department of Mathematics, Central University of Kerala, Kasaragod - 671316,\ Kerela,\ India.\ Email: shijintv11@gmail.com}
			Soorya P \footnote{\small  Department of Mathematics, Central University of Kerala, Kasaragod - 671316,\ Kerela,\ India.\ Email: sooryap2017@gmail.com}
			Germina K A \footnote{\small  Department of Mathematics, Central University of Kerala, Kasaragod - 671316,\ Kerela,\ India.\ Email: srgerminaka@gmail.com}
			Thomas Zaslavsky\footnote{Department of Mathematical Sciences, Binghamton University (SUNY), Binghamton, NY 13902-6000, U.S.A.  \textbf{e-mail:} {\tt zaslav@math.binghamton.edu}}
			
		}
\end{center}
\newcommand\spec{\operatorname{Spec}}
\thispagestyle{empty}
\begin{abstract}
Signed graphs have their edges labeled either as positive or negative.  Here we introduce two types of signed distance matrix for signed graphs. We characterize balance in signed graphs using these matrices and we obtain explicit formulae for the distance spectrum of some unbalanced signed graphs. We also introduce the notion of distance-compatible signed graphs and partially characterize it.
\end{abstract}



\section{Introduction}
A signed graph $\Sigma=(G,\sigma)$ is an underlying graph $G=(V,E)$ with a signature function $\sigma:E\rightarrow \{1,-1\}$. 
In this paper we introduce two types of signed distance for signed graphs and corresponding distance matrices, and we characterize balanced signed graphs using these matrices. We also explicitly compute the distance spectrum of some unbalanced signed graphs. Finally, we introduce the concept of distance compatibility for signed graphs and we prove two theorems, one of which characterizes compatible signed bipartite graphs. 

All the underlying graphs in our consideration are simple and connected, unless otherwise stated.

 Given a signed graph $\Sigma=(G,\sigma)$, the sign of a path $P$ in $\Sigma$ is defined as $\sigma(P)=\prod_{e\in E(P)} \sigma(e)$. We denote a shortest path between two given vertices $u$ and $v$ by $P_{(u,v)}$ and the collection of all shortest paths $P_{(u,v)}$ by $\mathcal{P}_{(u,v)}$; and $d(u,v)$ denotes the usual distance between $u$ and $v$.

\begin{defn}[Signed distance matrices]
We define two distance matrices for the signed graph $\Sigma$.  First we define auxiliary signs:
\par(S1) $\sigma_{\max}(u,v) = -1$ if all shortest $uv$-paths are negative, and $+1$ otherwise.
 \par(S2) $\sigma_{\min}(u,v) = +1$ if all shortest $uv$-paths are positive, and $-1$ otherwise.

 Then we define signed distances:
\par(d1) $d_{\max}(u,v) = \sigma_{\max}(u,v) d(u,v)=\max\{\sigma(P_{(u,v)}): P_{(u,v)} \in \mathcal{P}_{(u,v)} \}d(u,v).$
\par(d2) $d_{\min}(u,v) = \sigma_{\min}(u,v) d(u,v)=\min\{\sigma(P_{(u,v)}): P_{(u,v)} \in \mathcal{P}_{(u,v)} \}d(u,v).$

Finally, we define the signed distance matrices:
\par(D1)  $D^{\max}(\Sigma)=(d_{\max}(u,v))_{n\times n}$.
\par(D2)  $D^{\min}(\Sigma)=(d_{\min}(u,v))_{n\times n}$.
\end{defn}

The following are some immediate observations:

(U1) If there is a unique shortest $uv$-path $P_{(u,v)}$ in $G$, then $\sigma_{\max}(u,v)=\sigma_{\min}(u,v)=\sigma(P_{(u,v)})$.
Thus, for an edge $e=uv$, $\sigma_{\max}(u,v)=\sigma_{\min}(u,v)=\sigma(e)$.

(U2) For an underlying graph in which any two vertices are joined by a unique shortest path (this is called a geodetic graph), we have $\sigma_{\max}(u,v)=\sigma_{\min}(u,v)=\sigma(P_{(u,v)})$ and consequently $D^{\max}(\Sigma) = D^{\min}(\Sigma)$.  Some examples are signed graphs with underlying graph that is $K_n$ or a tree.

\begin{defn}
Two vertices $u$ and $v$ in a signed graph $\Sigma$ are said to be \emph{distance-compatible} (briefly, \emph{compatible}) if $d_{min}{(u,v)}=d_{max}{(u,v)}$.  And $\Sigma$ is said to be (distance-)compatible if every two vertices are compatible.
\end{defn}

\section{Balance}\label{mainsection}
In this section we mainly give different characterizations of balance in signed graphs using the signed distance matrices defined above. A signed graph $\Sigma=(G,\sigma)$ is said to be \emph{balanced} if the product $\sigma(C)=\prod_{e\in E(C)}\sigma(e)=1$ for all cycles $C$ in $\Sigma$. 
It is called \emph{antibalanced} if $-\Sigma$ is balanced; equivalently, every even cycle is positive and every odd cycle is negative.
Switching is an important operation in signed graphs which will be used very often in our discussion. Given a signed graph $\Sigma=(G,\sigma)$, we can construct another signed graph $\Sigma^\zeta=(G,\sigma^\zeta)$ where $\zeta:V\rightarrow\{1,-1\}$ and $\sigma^\zeta(uv)=\zeta(u)\sigma(uv)\zeta(v)$. We say then that $\Sigma$ is switched to $\Sigma^\zeta$. 
The adjacency matrix  $A(\Sigma)$ of a signed graph $\Sigma$ of order $n$ is defined as the real, symmetric square matrix $(a_{ij})$ of order $n$ where
$$a_{ij} =
\begin{cases}
\sigma(v_iv_j)  & \mbox{if } v_i\sim v_j ,\\
0 & \mbox{otherwise, }
\end{cases}
$$
with $v_i\sim v_j$ denoting adjacency. 
The adjacency spectrum of a signed graph consists of the eigenvalues of the adjacency matrix, counting multiplicities. Usually we use the notation $\begin{pmatrix} \lambda_1 &\lambda_2 &\cdots & \lambda_k\\
\alpha_1&\alpha_2&\cdots& \alpha_k\end{pmatrix}$ to denote that the spectrum of a matrix contains each of the elements in the first row as eigenvalues with the corresponding multiplicities in the second row.

If all edges of $\Sigma$ are positive, then $A(\Sigma)=A(G)$.  Therefore, for the adjacency matrix and eigenvalues, we can consider an unsigned graph to be an all-positive signed graph.

Several characterizations of balance are already available in the literature and we list some that are used in our discussion as follows.
\begin{thm}[Harary's bipartition theorem \cite{balance}]\label{Harary}
	A signed graph $\Sigma$ is balanced if and only if there is a bipartition of its vertex set, $V = V_1 \cup V_2$, such that every positive edge is induced by $V_1$ or $V_2$ while every negative edge has one endpoint in $V_1$ and one in $V_2$.
\end{thm}
A bipartition of $V$ as in Theorem \ref{Harary} is called a \emph{Harary bipartition} of $\Sigma$.
\begin{thm}[Harary's path criterion \cite{balance}]\label{Harary1}
	 A signed graph is balanced if and only if, for any two vertices $u$ and $v$, every $uv$-path has the same sign.
\end{thm}

For the main result about balance, Theorem \ref{Dbal}, we need Acharya's main theorem \cite{gbda}.  Since the proof can be made short, we restate it; that requires the weighted generalization of the Sachs formula for the characteristic polynomial of an adjacency matrix.  We assume the weight function $w$ on $E$ has values in a field but never takes the value $0$.  For a subgraph $B\subseteq G$, we define its weight as $w(B) = \prod_{e \in E(B)} w(e)$.  The adjacency matrix has the edge weight $w(uv)$ as the $(u,v)$ and $(v,u)$ elements and $0$ elsewhere; the characteristic polynomial of $(G,w)$ is defined as $\phi(G,w;\lambda) = \det(\lambda I - A(G,w))$.  An \emph{elementary subgraph} of $G$ is a subgraph that is a disjoint union of cycles and independent edges.  Let $\bfB_k(G)$ be the set of elementary subgraphs of order $k$.  We need a weighted Sachs formula, which can be found, e.g., in \cite{gbda, spec1}.  For an elementary graph $B$, let $B_2$ denote its set of isolated edges.

\begin{thm}[Weighted Sachs formula]\label{sachs}
Let $(G,w)$ be a weighted graph.  The characteristic polynomial $\phi(G,w;\lambda) = \sum_{k=0}^n a_{n-k}\lambda^k$ has coefficients given by 
$$
a_{n-k} = \sum_{B \in \bfB_k(G)} (-1)^{\kappa(B)} 2^{c(B)} w(B)w(B_2),
$$
where $\kappa(B)$ denotes the number of components and $c(B)$ denotes the number of cycles in $B$.
\end{thm}
\begin{proof}[Indication of Proof]
The proof is similar to that of the unweighted Sachs formula, as for example in \cite[Propositions 7.2, 7.3]{biggs}.
\end{proof}

\begin{thm}[{Restatement of Acharya \cite[Theorem 1]{gbda}}]\label{wtbal}
Let $\Sigma$ be a signed graph with underlying graph $G$ and let $w$ be a positive weight function on the edges.  Then $(\Sigma,w)$ is cospectral with $(G,w)$ if and only if\/ $\Sigma$ is balanced.
\end{thm}
\begin{proof}
We supply a concise restatement of Acharya's proof.  

If $\Sigma$ is balanced, it can be switched to be all positive; thus we may assume $\Sigma=G$ without changing the spectrum, from which it follows that $(\Sigma,w)$ and $(G,w)$ are cospectral.

Suppose $\Sigma$ is unbalanced.  By assumption there exists a negative cycle.  
The smallest order, $k$, of a negative cycle is the smallest order of a negative elementary subgraph, and all negative elementary subgraphs of order $k$ are cycles.  Hence, in the Sachs formulae the term of a negative elementary subgraph of order $k$ in $a_{n-k}(G,w)$ is $-2w(B)$ and in $a_{n-k}(\Sigma,w)$ it is the negative, $2w(B)$.  Comparing the coefficients, we see that
$$
a_{n-k}(G,w) - a_{n-k}(\Sigma,w) = -4 \sum_{B} (-1)^{\kappa(B)}  w(B),
$$
summed over negative $k$-cycles of $\Sigma$.  This difference is nonzero because all terms $w(B)$ are positive; therefore $a_{n-k}(\Sigma,w) > a_{n-k}(G,w)$, from which it follows that the matrices $A(\Sigma,w)$ and $A(G,w)$ have different characteristic polynomials.  Therefore, they are not cospectral.
\end{proof}

\begin{cor}[{Acharya's spectral criterion \cite[Corollary 1.1]{gbda}}]\label{Acharya}
A signed graph $\Sigma=(G,\sigma)$ is balanced if and only if the spectra of the adjacency matrices of $\Sigma$ and $G$ coincide.
\end{cor}

A beautiful strengthening was recently proved by Stani\'c.

\begin{thm}[Stani\'c's spectral criterion {\cite[Lemma 2.1]{stanic}}]\label{Stanic}
A signed graph $\Sigma=(G,\sigma)$ is balanced if and only if the largest eigenvalues of the adjacency matrices of $\Sigma$ and $G$ coincide.
\end{thm}

\begin{cor}\label{C:Stanic}
A signed graph with positively weighted edges, $(\Sigma,w)=(G,\sigma,w)$, is balanced if and only if the largest eigenvalues of the adjacency matrices $A(\Sigma,w)$ and $A(G,w)$ coincide.
\end{cor}

\begin{proof}
This is a corollary of the proof in \cite{stanic}, which generalizes directly.
\end{proof}

We can even weaken this to $\rho(A(\Sigma,w)) \geq \rho(A(G,\sigma,w))$, since $\leq$ is automatic due to positivity of the weights.

\begin{thm}[Switching criterion \cite{tz1}]\label{switching} A signed graph $\Sigma=(G,\sigma)$ is balanced if and only if it can be switched to an all positive signed graph.	
\end{thm}

Recall that we assume the signed graphs are connected.  We construct two complete signed graphs from the distance matrices $D^{\max}$ and $D^{\min}$ as follows.
\begin{defn}
	The associated signed complete graph $K^{D^{\max}}(\Sigma)$ with respect to $D^{\max}(\Sigma)$ is obtained by  joining the non-adjacent vertices of $\Sigma$ with edges having signs 
	\begin{equation*}
	\sigma(uv)= \sigma_{\max}(uv)
	\end{equation*}

	The associated signed complete graph $K^{D^{\min}}(\Sigma)$ with respect to $D^{\min}(\Sigma)$ is obtained by joining the non-adjacent vertices of $\Sigma$ with edges having signs 
	\begin{equation*}
	\sigma(uv)= \sigma_{\min}(uv)
	\end{equation*}
\end{defn}

It is easy to check that $\Sigma$ is compatible if and only if $-\Sigma$ is compatible.  More generally, $\Sigma$ and $-\Sigma$ have the same pairs of compatible vertices. It follows that every antibalanced graph is compatible.
Thus, we have three kinds of compatible signed graphs: balanced, antibalanced, and geodetic.
We have the following situation.  All connected signed graphs fall into three classes:
\begin{itemize}
	\item [(I)] $\Sigma$ is balanced or antibalanced or geodetic, which makes $D^{\max}=D^{\min}$ (see Theorem~\ref{gen}).
	\item [(II)] $\Sigma$ is unbalanced and neither antibalanced nor geodetic, but still $D^{\max}=D^{\min}$.
	\item [(III)] $\Sigma$ is unbalanced and neither antibalanced nor geodetic, with $D^{\max}\neq D^{\min}$.
\end{itemize}
 In the first two cases, we denote the associated signed complete graph by $K^{D^{\pm}}(\Sigma)$ or simply by $K^{D^{\pm}}$.

 A small example for case (II) is the triangle with one negative edge. 
As examples for all three cases consider the signed graph $K_4 \setminus \{e\}$ with three different sign patterns as given in Figure \ref{Fthree}.  Note that this graph is not geodetic.

\begin{figure}[h]
	\centering
	\includegraphics[width=14cm]{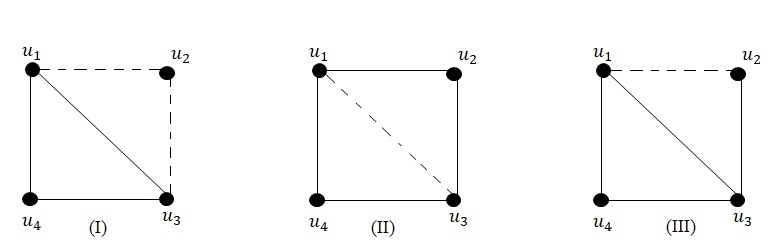}
	\caption{Three different kinds of signature of $K_4 \setminus \{e\}$.}
	\label{Fthree}
\end{figure}

The signed graph (I) in Figure \ref{Fthree} is balanced and has $D^{\max}=D^{\min}$; the matrix is
$$D^{\max}=D^{\min}=\begin{pmatrix}
0 & -1 & 1 & 1\\
-1 & 0 & -1 &-2\\
1 & -1 & 0 & 1\\
1 & -2 & 1 & 0
\end{pmatrix}$$
The signed graph (II) in Figure \ref{Fthree} is neither balanced nor antibalanced, yet $D^{\max}=D^{\min}$, as follows:
 $$D^{\max}=D^{\min}=\begin{pmatrix}
0 & 1 & -1 & 1\\
1 & 0 & 1 &2\\
-1 & 1 & 0 & 1\\
1 & 2 & 1 & 0
\end{pmatrix}.$$
The signed graph (III) in Figure \ref{Fthree} is of type (III), where $D^{\max}\ne D^{\min}$. The two matrices are
 $$D^{\max}=\begin{pmatrix}
0 & -1 & 1 & 1\\
-1 & 0 & 1 & 2\\
1 & 1 & 0 & 1\\
1 & 2 & 1 & 0
\end{pmatrix},
\qquad
D^{\min}=\begin{pmatrix}
0 & -1 & 1 & 1\\
-1 & 0 & 1 & -2\\
1 & 1 & 0 & 1\\
1 & -2 & 1 & 0
\end{pmatrix}.$$

Figure \ref{Figuren5} shows another simple, $2$-connected, compatible signed graph that is neither geodetic nor balanced nor antibalanced.%
\begin{figure}[h]
	\centering
	\includegraphics[width=3.5cm]{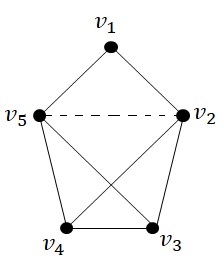}
	\caption{A nontrivially distance-compatible signed graph of order $5$.}
	\label{Figuren5}
\end{figure}

\section{Some new characterizations for balance}
\begin{thm} \label{gen}
	For a signed graph $\Sigma$ the following statements are equivalent:
	\begin{enumerate}
	\item [\rm{(i)}] $\Sigma$ is balanced.
	\item [\rm{(ii)}] The associated signed complete graph $K^{D^{\max}}(\Sigma)$ is balanced.
	\item [\rm{(iii)}] The associated signed complete graph $K^{D^{\min}}(\Sigma)$ is balanced.
	\item [\rm{(iv)}] $D^{\max}(\Sigma)=D^{\min}(\Sigma)$ and the associated signed complete graph $K^{D^{\pm}}(\Sigma)$ is balanced.
	\end{enumerate}	
\end{thm}
\begin{proof}
	Suppose that $\Sigma=(G,\sigma)$ is a balanced signed graph. Then, there exists a Harary bipartition $V_1$ and $V_2$ of the vertex set $V(\Sigma)$. Let $u,v$ be any two vertices of $\Sigma$, then using Theorem~\ref{Harary1}, $D^{\max}(\Sigma)=D^{\min}(\Sigma)$.  Thus, $D^\pm(\Sigma)$ is well defined.
	Now to obtain the associated signed complete graph $K^{D^{\pm}},$ we join the non-adjacent vertices of $\Sigma$ with signs $\sigma(uv)= \sigma_{\max}(uv)= \sigma_{\min}(uv)$.   Then $V_1$ and $V_2$ will form a Harary bipartition of $K^{D^{\pm}}$. Thus, $K^{D^{\pm}}$ is balanced, per Theorem~\ref{Harary1}.
	
	 Conversely, suppose that $D^{\max}(\Sigma)$ and the associated signed complete graph $K^{D^{\max}}$ is balanced. Then $\Sigma$, being a subgraph of $K^{D^{\max}},$  is balanced.  Similarly, balance of $K^{D^{\min}}$ implies balance of $\Sigma$.
\end{proof}

 Using Acharya's spectral criterion in Corollary~\ref{Acharya} and Stani\'c's criterion in Corollary \ref{C:Stanic}, the above characterization can be reformulated as
\begin{thm}
	A signed graph $\Sigma$ is balanced if and only if  the associated signed complete graph $K^{D^{\pm}}(\Sigma)$ has largest eigenvalue $n-1$, equivalently its spectrum is
	$\left(\begin{smallmatrix}  n-1 & -1\\ 1 & n-1
	\end{smallmatrix}\right)$.
	\end{thm}
	
 Now we provide one more characterization in terms of the signed distance matrices. 
 \begin{lem}\label{Lswitch}
 Switching a signed graph $\Sigma$ does not change the set of compatible pairs of vertices.
 \end{lem}
 \begin{proof}
Consider two vertices, $u$ and $v$. Switching changes the sign of all $uv$-paths, or none.  Therefore, if $u$ and $v$ are compatible in $\Sigma$, they are compatible in the switched graph, and conversely.
 \end{proof}
 
\begin{thm}\label{swich}
	If $\Sigma$ is switched to $\Sigma^\zeta$ and if $D^{\max}(\Sigma)=D^{\min}(\Sigma)=D^{\pm}(\Sigma)$ then $D^{\max}(\Sigma^\zeta)=D^{\min}(\Sigma^\zeta)=D^{\pm}(\Sigma^\zeta)$ and $D^{\pm}(\Sigma)$ is similar to $D^{\pm}(\Sigma^\zeta)$.
\end{thm}
\begin{proof}
The first part follows from Lemma \ref{Lswitch}.  Also an easy calculation shows that the switching matrix $S=\diag(\zeta(v_1)), \zeta(v_2)),\ldots, \zeta(v_n))=S^{-1}$ makes $D^{\pm}(\Sigma^\zeta)=SD^{\pm}(\Sigma)S=SD^{\pm}(\Sigma)S^{-1}$.
\end{proof}

\begin{thm}\label{Dbal}
The following properties of a signed graph $\Sigma$ are equivalent.
\begin{enumerate}
\item [\rm{(i)}] $\Sigma$ is balanced.
\item [\rm{(ii)}] $D^{\max}(\Sigma)$ is cospectral with $D(G)$.
\item [\rm{(iii)}] $D^{\min}(\Sigma)$ is cospectral with $D(G)$.
\item [\rm{(iv)}] $D^{\max}(\Sigma)$ has largest eigenvalue equal to that of $D(G)$.
\item [\rm{(v)}] $D^{\min}(\Sigma)$ has largest eigenvalue equal to that of $D(G)$.
\end{enumerate}
In particular, $\Sigma$ is balanced if and only if $D^\pm(\Sigma)$ exists and is cospectral with $D(G)$.
\end{thm}
\begin{proof}
Consider $K_n$ with the weight function $w(uv) = d_G(u,v)$.  By Theorem \ref{wtbal}, $(K^{D^{\max}}(\Sigma),w)$ is cospectral with $(K_n,w)$.  Part (ii) is this statement in terms of adjacency matrices.  Part (iv) follows similarly from Corollary \ref{C:Stanic}.

The same argument applies for $D^{\min}(\Sigma)$.
\end{proof}

We illustrate the use of Theorem \ref{Dbal} with the example of a complete bipartite graph $K_{n,n}$.

\begin{cor}
	The signed graph $\Sigma=(K_{n,n},\sigma)$ is balanced if and only if $D^{\max}$ (or $D^{\min}$) has largest eigenvalue $3n-1$; also if and only if its spectrum is 
	$\left(\begin{smallmatrix}  3n-2 & n-2 &-2\\ 1 & 1 & 2n-2
	\end{smallmatrix}\right).$
\end{cor}

\section{Examples of signed-distance spectra}
We prepare for the examples with a sum formula.

\begin{lem}\label{sum}
	 $\sum_{r=1}^{k}rz^r=k\frac{z^{k+1}}{z-1}-\frac{z(z^k-1)}{(z-1)^2}$. Hence, 
	 $$\sum_{r=1}^{k}r\cos(r\theta)=\frac{1}{2}\Big(\frac{k\sin((2k+1)\theta/2)}{\sin(\theta/2)} -\frac{\sin^2(k\theta/2)}{\sin^2(\theta/2)}\Big).$$
\end{lem}
\begin{proof}
	Taking $S(z)=\sum_{r=1}^{k}rz^r=k\frac{z^{k+1}}{z-1}-\frac{z(z^k-1)}{(z-1)^2}$ and putting $z=e^{i\theta}$, $$S(e^{i\theta})=k\frac{e^{i(k+1)\theta}}{e^{i\theta}-1}-\frac{e^{i\theta}(e^{ik\theta}-1)}{(e^{i\theta}-1)^2}
	=\frac{i}{2} \Big(-\frac{e^{i(k+1)\theta/2}}{\sin(\theta/2)}+ \frac{e^{ik\theta/2}\sin{k\theta/2}}{\sin^2(\theta/2)}\Big ).
$$ 
	From this $\sum_{r=1}^{k}r\cos(r\theta)=\text{Re}\, S(e^{i\theta})$, which equals the stated formula.
\end{proof}

 For the odd unbalanced cycle $C^{-}_{n}$ where $n=2k+1$, we have unique shortest path between any two vertices. Thus $D^{\max}(C^{-}_{n})=D^{\min}(C^{-}_{n})=D^{\pm}(C^{-}_{n})$ and we give the spectrum of an odd unbalanced cycle as follows.
 
\begin{thm}
	For an odd unbalanced cycle $C^{-}_{n}$ where $n=2k+1$, the spectrum of $D^\pm$ is 
	$$\begin{pmatrix}
	k(-1)^k-\dfrac{1-(-1)^k}{2}  & \dfrac{k(-1)^j}{\sin((2j+1)\frac{\pi}{2n})}-\dfrac{\sin^2((2j+1)\frac{k\pi}{2n})}{\sin^2((2j+1)\frac{\pi}{2n})} \\
	1 & 2 \quad (j=0,1,2,\ldots, k-1)
\end{pmatrix}.$$
\end{thm}
\begin{proof}
First we note that by a proper switching any unbalanced cycle $C^{-}_{2k+1}=v_1v_2\cdots v_{2k}v_{2k+1}v_1$ can be switched to a cycle with only one negative edge, say $\sigma(v_{2k+1}v_1)=-1$ and in light of Theorem~\ref{swich}, we only have to provide the proof of this switched cycle. The distance matrix $D^{\pm}(C^{-}_{2k+1})$ can be written as
$$D^{\pm}(C^{-}_{2k+1})=\begin{pmatrix}
	0 & 1 & 2 & \cdots k & -k & -(k-1) \cdots & -2 & -1\\
	1 & 0 & 1 & 2 \cdots k &-k & \cdots & \cdots & -2\\
	\cdots & \cdots & \cdots & \cdots \cdots &\cdots& \cdots & \cdots &\cdots\\
	\cdots & \cdots & \cdots & \cdots \cdots &\cdots& \cdots & \cdots &\cdots\\	 
	-1 & -2 & \cdots & \cdots   &\cdots & \cdots& 1& 0
\end{pmatrix}
$$
This is a real symmetric matrix and hence its eigenvalues are real.	
Choosing $\rho=e^{i\theta}$ and the eigenvector corresponding to the eigenvalue $\lambda$ as $X=[\rho,\rho^2,\ldots,\rho^{n}]^T$ we get from the equation $D^{\pm}(C^{-}_{n})X=\lambda X$, $\lambda=\sum_{r=1}^{k}r\rho^{r}-\sum_{r=1}^{k}r\rho^{n-r}$. Choosing now $\rho^n=-1$ so that $\rho=e^{i\theta_j}=e^\frac{{(2j+1)}i\pi}{n}$ for $j=0,1,\ldots,n-1$, we get	 $\lambda_{j}=\sum_{r=1}^{k}r\rho^{r}+ \sum_{r=1}^{k}r\rho^{-r}= 2\sum_{r=1}^{k}r\cos{(r\theta_{j})}$. Using Lemma~\ref{sum},
\begin{equation}\label{lambdj}
\lambda_j=\frac{k\sin((2k+1)\theta_j/2)}{\sin(\theta_j/2)} -\frac{\sin^2(k\theta_j/2)}{\sin^2(\theta_j/2)}.
\end{equation} 
Recall that here $\theta_{j}=\frac{{(2j+1)}\pi}{2k+1}$ for $j=0,1,2,\ldots, 2k$. For $j=k$, we have $\theta_{k}=\frac{{(2k+1)}\pi}{2k+1}=\pi$ and $\rho_k=-1$ and hence from Equation~\eqref{lambdj}, $\lambda_k=k(-1)^k-\frac12[1-(-1)^k]$, which is either $k$ or $-(k+1)$ according as $k$ is even or odd. The remaining eigenvalues can be paired as
$\lambda_{2k}=\lambda_{0}, \lambda_{2k-1}=\lambda_{1},\ldots, \lambda_{k+1}=\lambda_{k-1}$. For $j=0,1,\ldots, k-1$, from Equation~\eqref{lambdj}
$$\lambda_j= \dfrac{k(-1)^j}{\sin((2j+1)\frac{\pi}{2n})}-\dfrac{\sin^2((2j+1)\frac{k\pi}{2n})}{\sin^2((2j+1)\frac{\pi}{2n})},$$
each of which repeats twice due to pairing.
\end{proof}

Now we consider an unbalanced signed wheel $(W_{n+1},\sigma)=(C_n\vee K_1,\sigma)$ where $n$ is odd and the signature $\sigma$ is such that $\sigma(e)=-1$ if $e\in E(C_n)$ and $1$ otherwise. Clearly, this signed graph is compatible (in fact, it is antibalanced).

\begin{thm}
The spectrum of $(W_{n+1},\sigma)$ is
	$$\spec D^{\pm}(W_{n+1},\sigma)=\begin{pmatrix}
	n-4\pm\sqrt{n^2-7n+16} & -2-6\cos\dfrac{2j\pi}{n}\\
	1 & 1 \ (j=1,2,\ldots, n-1)
	\end{pmatrix}.
	$$
\end{thm}
\begin{proof}
	It can be seen that $A(C_n^{-})=-A(C_n)$ and the signed distance matrix of the signed wheel is $D^{\pm}=\begin{pmatrix}
	2A(\overline{C_n})-A(C_n) & J_{n\times 1}\\
	J_{1\times n} & 0
	\end{pmatrix}$ where $J$ is the all-ones matrix. The eigenvalues of $D^{\pm}$ are the eigenvalues of $2A(\overline{C_n})-A(C_n) = 2J_{n\times n} - 2I -3A(C_n)$ whose eigenvectors are orthogonal to $J_{n\times1}$ and those of the equitable quotient matrix $M=\begin{pmatrix}
	2n-8 & 1\\
	n & 0
	\end{pmatrix}$. 
	The eigenvalues of $A(C_n)$ are well known to be $\lambda_j=2\cos\dfrac{2j\pi}{n}$ for $j=0,1,2,\ldots, n-1$ with $\lambda_0 = 2$ corresponding to the eigenvector $J_{n\times1}$ and the other eigenvectors orthogonal to $J_{n\times1}$.
	The relevant eigenvalues of $2A(\overline{C_n})-A(C_n)$ are $-3\lambda_j-2= -2-6\cos\dfrac{2j\pi}{n}$ for $j=1,2,\ldots,n-1$. The eigenvalues of $M$ are $n-4\pm\sqrt{n^2-7n+16},$ which completes the proof.
\end{proof}

\section{Compatible and incompatible signed graphs}\label{section3}
We know three kinds of compatible signed graph: balanced, antibalanced, and geodetic, and there are others we do not know.  For bipartite graphs we have better information.

\begin{thm}\label{bipartite}
	A bipartite signed graph is distance-compatible if and only if it is balanced.
\end{thm}

\begin{proof}
Let $\Sigma$ be a compatible bipartite signed graph and suppose it is unbalanced.  Then there is a shortest negative cycle, say $C = v_0v_1 \cdots v_k v_{k+1} \cdots v_{2k}$, of length $2k$, where $v_0=v_{2k}$.  In $C$ there are two $v_0v_k$-paths of opposite sign, say they are $P = v_0v_1 \cdots v_k$ and $Q = v_{2k} \cdots v_{k+1}v_k$.  Since $P$ and $Q$ have different signs, there is a shorter $v_0v_k$-path, $R$, of length $l$.  By switching and choice of notation we may assume $P$ and $R$ are positive and $Q$ is negative.  Now, the concatenation $Q R^{-1}$ is a negative closed walk, so it contains a negative cycle, $C'$.  But the length of $QR^{-1}$ is $k+l < 2k$, so the length of $C'$ is at most $k+l < 2k$, the length of $C$.  Thus, $C'$ is a shorter negative cycle than $C$, a contradiction.  It follows that all cycles are positive.
	
The converse is known from Theorem \ref{Harary1}.
\end{proof}

\begin{cor}\label{biD}
Let $\Sigma$ be a bipartite signed graph. Then, $\Sigma$ is balanced if and only if $D^{\max}(\Sigma)=D^{\min}(\Sigma).$
\end{cor}

\begin{thm}\label{blockreduction}
	A signed graph is compatible if and only every block of it is compatible.
\end{thm}

\begin{proof}
	The ``only if'' is trivial.
	
	Suppose every block of $\Sigma$ is compatible.  Consider two vertices, $u$ and $v$, and two shortest $uv$-paths, $P$ and $Q$ (if they exist).  If $u$ and $v$ are in the same block, we are done, since all $uv$-paths lie within that block.  If not, consider the block/cutpoint tree $T$ of $\Sigma$.  Any path from $u$ to $v$ starts at a block $B_u$ that contains $u$ and ends at a block $B_v$ that contains $v$.  (The block $B_u$ is unique if $u$ is not a cutpoint.  If $u$ is a cutpoint, let $B_u$ be the nearest block to $v$ in the cutpoint tree.  The same for $B_v$.)  There is a unique path from $B_u$ to $B_v$ in $T$.  Let that path be $B_u=B_0,c_1, B_1,c_2, \ldots,c_l, B_l = B_v$, where $c_i$ is the cutpoint at which $B_{i-1}$ and $B_i$ intersect.  Both $P$ and $Q$ must go through all these blocks in order; so we can split them into component paths, $P_i = P\cap B_i$ and $Q_i = Q \cap B_i$.  Then $P = P_0c_1P_1c_2 \cdots c_l P_l$ and $Q = Q_0c_1Q_1c_2 \cdots c_l Q_l$.
	
	Now, $P_0$ is a shortest path connecting $u$ to $c_1$, $P_i$ is a shortest path from $c_{i}$ to $c_{i+1}$ for $1 \leq i < l$, and $P_l$ is a shortest path from $c_l$ to $v$.  (If they were not shortest, then $P$ would not be shortest.)  Similarly, each segment $Q_i$ is a shortest path with the same endpoints as $P_i$.  By compatibility of the blocks, $\sigma(P_i) = \sigma(Q_i)$ for all $i$; therefore $\sigma(P) = \sigma(Q)$.  Thus, $\Sigma$ is compatible.
\end{proof}

The characterization problem for all compatible signed graphs is reduced by these theorems to the case of a 2-connected, non-bipartite graph that is not a cycle.  That appears to be a hard problem.  We know there exist examples that are not any of the three kinds stated at the beginning of this section; we conclude with such an example.  We are presently engaged in research attempting to solve this problem.

\begin{ex}\label{counterex}
Let $P$ and $Q$ be graphs of orders $p$ and $q$, respectively, that are 2-connected, incomplete, and not bipartite.  Consider $P$ as an all-positive signed graph and $Q$ as an all-negative signed graph.  Let $K_{p,q}$ be the complete bipartite graph considered as an all-positive signed graph, with partition sets $A$ and $B$ of orders $p$ and $q$, respectively.  Now identify $V(P)$ with $A$ and $V(Q)$ with $B$; this gives a signed graph $\Sigma$.  We prove that $\Sigma$ is not balanced, not antibalanced, and not geodetic, but it is distance-compatible.

Because $P$ contains a positive odd cycle and $Q$ contains a negative odd cycle, $\Sigma$ is neither balanced nor antibalanced.  The diameter of $\Sigma$ is 2.  Since $P, Q$ are 2-connected, $p, q \geq 3$; hence any two non-adjacent vertices in $P$ (or in $Q$) are joined by multiple shortest paths of length 2, all of which are positive.  (A 2-path in $K_{p,q}$ or $P$ is positive.  A 2-path in $Q$ is positive because $Q$ is all negative.)  It follows that $\Sigma$ is compatible but it is not geodetic.
\end{ex}

We note that there are many graphs suitable to be $P$ and $Q$, such as the complemented line graphs of complete graphs, $\overline{L(K_n)}$ for $n\geq 5$ ($n=5$ being the Petersen graph).

\section*{Acknowledgements}

The second and third authors would like to acknowledge their gratitude to the Council of Scientific and Industrial Research (CSIR), India, for the financial support under the CSIR Junior Research Fellowship scheme, vide order nos.: 09/1108(0032)/2018-EMR-I and 09/1108(0016)/2017-EMR-I, respectively. The fourth author would like to acknowledge her gratitude to Science and Engineering Research Board (SERB), Govt.\ of India, for the financial support under  the scheme Mathematical Research Impact Centric Support (MATRICS), vide order no.: File No. MTR/2017/000689.  
The last author thanks the SGPNR for support under grant no.\ 20A-01079.

\end{document}